\documentclass{article}
\title{Quasi-coherent sheaves on complex analytic spaces}
\usepackage{amsmath,amssymb,amsthm,tikz-cd}
\usepackage[backend=bibtex, style=alphabetic]{biblatex}
\DeclareFieldFormat{postnote}{\mknormrange{#1}}
\DeclareFieldFormat{volcitepages}{\mknormrange{#1}}
\DeclareFieldFormat{multipostnote}{\mknormrange{#1}}
\usepackage[hidelinks]{hyperref}
\newcommand\setItemnumber[1]{\setcounter{enumi}{\numexpr#1-1\relax}}
\newtheorem{thm}{Theorem}[section]
\newtheorem{lm}[thm]{Lemma}
\newtheorem{ft}[thm]{Fact}
\theoremstyle{definition}
\newtheorem{df}[thm]{Definition}
\newtheorem{eg}[thm]{Example}
\def\C{\mathbb{C}}
\def\cH{\mathcal{H}}
\DeclareMathOperator\coker{coker}
\DeclareMathOperator\colim{colim}
\DeclareMathOperator\Hom{Hom}
\DeclareMathOperator\Id{Id}
\DeclareMathOperator{\Mod}{Mod}
\DeclareMathOperator{\Qch}{Qch}
\usepackage{authblk,orcidlink}
\author{Haohao \textsc{Liu}\,\orcidlink{0009-0007-5942-8174}\\ email: \href{mailto:kyung@mail.ustc.edu.cn}{kyung@mail.ustc.edu.cn}}
\affil{Institut de Mathématiques de Jussieu-Paris Rive Gauche, 4 place Jussieu,  Paris  75005, France}
\date{\today}
\begin{document}
	\maketitle
		\begin{abstract}
		We show that in the category of analytic sheaves on a complex analytic space, the full subcategory of quasi-coherent sheaves  is  an abelian subcategory.  
	\end{abstract}
\section{Introduction}
Let $(X,O_X)$ be a ringed space. The category of $O_X$-modules is denoted by $\Mod(O_X)$.
\begin{df}An $O_X$-module $F$ is called \emph{quasi-coherent} if for  every $x\in X$, there is an open neighborhood $U\subset X$ of $x$, two sets $I,J$  and  a morphism $O_U^{\oplus J}\to O_U^{\oplus I}$ with cokernel  isomorphic to $F|_U$. The full subcategory of $\Mod(O_X)$  of quasi-coherent modules is denoted by $\Qch(X)$.\end{df} According to \cite[\href{https://stacks.math.columbia.edu/tag/01BD}{Tag 01BD}]{stacks-project}, in general $\Qch(X)$ is not an abelian category. By \cite[\href{https://stacks.math.columbia.edu/tag/06YZ}{Tag 06YZ}]{stacks-project}, if $X$ is a scheme, then  $\Qch(X)$ is a weak Serre subcategory (in the sense of \cite[\href{https://stacks.math.columbia.edu/tag/02MO}{Tag 02MO~(2)}]{stacks-project}) of $\Mod(O_X)$.   We show a complex analytic analog of this result, contrary to a guess made in \cite{421073}.
\begin{thm}\label{thm:qchabel}If $X$ is a complex analytic space, then $\Qch(X)\subset \Mod(O_X)$ is a weak Serre subcategory. In particular, it is an abelian subcategory.\end{thm}

\section{Preliminaries}
Let $(X,O_X)$ be a ringed space. Every $O_X(X)$-module induces naturally a quasi-coherent $O_X$-module.
\begin{eg}[{\cite[\href{https://stacks.math.columbia.edu/tag/01BI}{Tag 01BI}]{stacks-project}}]\label{eg:inducedqch}
	Let $f:(X,O_X)\to (\{*\},O_X(X))$ be the morphism of ringed spaces, with $f:X\to \{*\}$ the unique map and  $f_*^{\natural}:O_X(X)\to O_X(X)$ the identity. Then $f$ is  flat. For an $O_X(X)$-module $M$, its pullback $f^*M$ is called the sheaf associated with $M$. This $O_X$-module is quasi-coherent. The functor $f^*:\Mod(O_X(X))\to \Mod(O_X)$ is called  \emph{localization} and denoted by $\tilde{\cdot}$.
\end{eg}

From \cite[4.1.1]{EGA1}, on a scheme the direct sum of any family of quasi-coherent modules is quasi-coherent. It fails for complex manifolds, shown by Example \ref{eg:Starr}.  
\begin{eg}\cite{453047}\label{eg:Starr}
	Let $X\subset \C$ be the  unit open  disk. For every integer $n\ge 2$, Gabber  (\cite[Eg.~2.1.6]{conrad2006relative})  constructs a locally free (hence quasi-coherent) $O_X$-module $F_n$ of infinite rank, such that for every open subset $U\subset X$ containing $\{\pm 1/n\}$, one has $\Gamma(U,F_n)=0$. We prove that $F:=\oplus_{n\ge 2}F_n$ is not quasi-coherent.
	
	Assume the contrary. Then there is an open neighborhood $V$ of $0\in X$, a set $I$  and  a  quotient morphism $q:O_V^{\oplus I}\to F|_V$.  There is an integer $N\ge 2$ with $\{\pm 1/N\}\subset V$. Let $p:F|_V\to F_N|_V$ be the quotient morphism. Because $\Hom_{\Mod(O_V)}(O_V,F_N|_V)=\Gamma(V,F_N)=0$, the morphism $pq=0$. However, it contradicts $F_N|_V\neq0$. 
\end{eg}
Let $X$ be a complex analytic space  in the sense of \cite[p.18]{grauert2013theory}. For an inclusion $i:K\to X$ of a compact subset, let $O_K=i^{-1}O_X$. Then $O_K$ is naturally a sheaf of rings on $K$.
\begin{df}
	A compact subset $K\subset X$ is  a \emph{Stein compactum}, if $K$ has a fundamental system of open neighborhoods that are Stein subspaces of $X$. A Stein compactum $K$ is  \emph{Noetherian} if $O_K(K)$ is a Noetherian ring.
\end{df}
\begin{ft}[{\cite[Thm.~I, 9; Rem.~I, 10]{frisch1967points}}]\label{ft:I10}
	Every $x\in X$ admits a neighborhood which is a Noetherian Stein compactum in $X$.
\end{ft}
\begin{lm}\label{lm:qch} Let $F$ be an $O_X$-module. Then the following conditions are equivalent:
	\begin{enumerate}
		\item\label{it:Bqch} \textup{(\cite[Def.~5.1]{ben2007non})}  Every $x\in X$ admits a  neighborhood $K$  which is a Noetherian Stein compactum,  such that $F|_K$ is associated with  an $O_K(K)$-module.
		\item\label{it:EGAqch} The $O_X$-module $F$ is quasi-coherent.
	\end{enumerate} 
\end{lm}
\begin{proof}
	\hfill	\begin{itemize}
		\item Assume Condition \ref{it:Bqch}. For every $x\in X$, take such a $K$ and suppose that $F|_K$ is associated with  an $O_K(K)$-module $M$.  There are sets $I$, $J$ and an exact sequence $O_K(K)^{\oplus I}\to O_K(K)^{\oplus J}\to M\to0$ in the category of $O_K(K)$-modules. By \cite[\href{https://stacks.math.columbia.edu/tag/01BH}{Tag 01BH}]{stacks-project}, it induces an exact sequence $O_K^{\oplus I}\to O_K^{\oplus J}\to F|_K\to0$ in $\Mod(O_K)$.  Then the $O_{K^{\circ}}$-module $F|_{K^{\circ}}$ is quasi-coherent. Thus, Condition \ref{it:EGAqch} is proved.
		\item Assume Condition \ref{it:EGAqch}. Because $X$ is  locally compact Hausdorff, for every $x\in X$, by \cite[\href{https://stacks.math.columbia.edu/tag/01BK}{Tag 01BK}]{stacks-project}, there is an open neighborhood $U\subset X$ of $x$ such that $F|_U$ is associated with a $\Gamma(U,O_X)$-module. From Fact \ref{ft:I10}, there is a  neighborhood $K$   of $x\in U$ which is a Noetherian Stein compactum. By  \cite[\href{https://stacks.math.columbia.edu/tag/01BJ}{Tag 01BJ}]{stacks-project} applied to the morphism $(K,O_K)\to (U,O_U)$ of ringed spaces,   $F|_K$ is associated with an $O_K(K)$-module. Thus, Condition \ref{it:Bqch} is proved.
	\end{itemize}
\end{proof}
\begin{lm}\label{lm:locK}
	Let	$K$  be a Noetherian Stein compactum in  $X$.
	\begin{enumerate}
		\item\label{it:locsec}   The natural transformation $\Id\to \Gamma(K,\tilde{\cdot})$ of functors  $\Mod(O_K(K))\to \Mod(O_K(K))$ is an isomorphism.
		\item\label{it:fullfaithful} The localization functor $\tilde{\cdot}:\Mod(O_K(K))\to \Mod(O_K)$ is exact, fully faithful.
		\item\label{it:locvanish} For every $O_K(K)$-module $M$ and every integer $q>0$, one has $H^q(K,\tilde{M})=0$.
	\end{enumerate}
\end{lm}
\begin{proof}
	\hfill	\begin{enumerate}\item Let $M$ be an $O_K(K)$-module. We prove that the  morphism $M\to \Gamma(K,\tilde{M})$ is  an isomorphism. Assume first that $M$ is finitely generated.  Then the result follows from \cite[p.299]{taylor2002several}. Assume now that $M$ is arbitrary. Let $\{M_i\}_{i\in I}$ be the  family of all finitely generated submodules of $M$.   This family is directed in the inclusion relation and \begin{equation}\label{eq:sumfg}M=\sum_{i\in I}M_i.\end{equation} By \cite[\href{https://stacks.math.columbia.edu/tag/01BH}{Tag 01BH (4)}]{stacks-project}, the localization functor preserves colimits. Therefore, \begin{equation}\label{eq:locsumfg}\tilde{M}=\colim_{i\in I}\tilde{M_i}.\end{equation} By \cite[Thm.~4.12.1]{godement1958topologie}, one has \[\Gamma(K,\tilde{M})=\colim_{i\in I}\Gamma(K,\tilde{M_i})=\colim_{i\in I}M_i=M.\] 
		\item The exactness is proved in \cite[Prop.~11.9.3 (ii)]{taylor2002several}. For any $M,N\in \Mod(O_K(K))$, we prove that the natural morphism \begin{equation}\label{eq:lochom}\Hom_{O_K(K)}(M,N)\to \Hom_{O_K}(\tilde{M},\tilde{N})\end{equation} is an isomorphism. 
		
		Assume first that $M$ is finitely generated. As the ring $O_K(K)$ is Noetherian, the $O_K(K)$-module $M$ is of finite presentation.  Then by \cite[Exercise~7.20~(b)]{gortz2020algebraic}, one has $\widetilde{\Hom_{O_K(K)}(M,N)}=\cH om_{O_K}(\tilde{M},\tilde{N})$. By Point \ref{it:locsec}, the morphism (\ref{eq:lochom}) is an isomorphism. Assume now that $M$ is arbitrary. By (\ref{eq:sumfg}) and (\ref{eq:locsumfg}), the morphism (\ref{eq:lochom}) is the inverse limit of the morphisms $\Hom_{O_K(K)}(M_i,N)\to \Hom_{O_K}(\tilde{M_i},\tilde{N})$, each of which is an isomorphism.
		\item When $M$ is finitely generated, it follows from \cite[Prop.~11.9.2]{taylor2002several} and \cite[Thm.~1 (B)]{cartan1957varietes}. Assume now that $M$ is arbitrary. By (\ref{eq:locsumfg}) and \cite[Thm.~4.12.1]{godement1958topologie}, one has $H^q(K,\tilde{M})=\colim_iH^q(K,\tilde{M_i})=0$.\end{enumerate}
\end{proof}
\section{Proof of Theorem \ref{thm:qchabel}}

\begin{enumerate}\item\label{it:kercokercls}	For every morphism $f:F\to G$ in $\Qch(X)$, we prove that $\ker(f),\coker(f)$ in $\Mod(O_X)$ lie in $\Qch(X)$.\end{enumerate}

For every $x\in X$, by Lemma \ref{lm:qch}, there is a  neighborhood $A$ (resp. $B$) of $x\in X$ which is a Noetherian Stein compactum and an $O_A(A)$-module $M$ (resp.  $O_B(B)$-module $N$), such  that $F|_A$ (resp. $G|_B$) is associated with $M$ (resp. $N$). 	By   Fact \ref{ft:I10}, there is a neighborhood $C$ of $x\in A^{\circ}\cap B^{\circ}$ which is a Noetherian Stein compactum. From \cite[\href{https://stacks.math.columbia.edu/tag/01BJ}{Tag 01BJ}]{stacks-project}, $F|_C$ (resp. $G|_C$) is associated with $M\otimes_{O_A(A)}O_C(C)$ (resp. $N\otimes_{O_B(B)}O_C(C)$). By Lemma \ref{lm:locK} \ref{it:fullfaithful}, there is a morphism  \[\phi:M\otimes_{O_A(A)}O_C(C)\to N\otimes_{O_B(B)}O_C(C)\] in $\Mod(O_C(C))$ whose localization is $f|_C:F|_C\to G|_C$.   The restriction functor $\Mod(O_X)\to \Mod(O_{C^{\circ}})$ is exact,  so $\ker(f)|_{C^{\circ}}$ (resp. $\coker(f)|_{C^{\circ}}$) is the localization of $\ker(\phi\otimes_{O_C(C)}\Id_{O_X(C^{\circ})})$ (resp. $\coker(\phi\otimes_{O_C(C)}\Id_{O_X(C^{\circ})})$) in $\Mod(O_X(C^{\circ}))$. Therefore, the $O_X$-modules $\ker(f),\coker(f)$ are quasi-coherent.  

\begin{enumerate}\setItemnumber{2}\item\label{it:extensioncls}	Let \begin{equation}\label{eq:extinQch}0\to F'\to F\to F''\to0\end{equation} be a short exact sequence in $\Mod(O_X)$, with $F',F''$ quasi-coherent. We prove that $F$ is quasi-coherent.\end{enumerate} By Lemma \ref{lm:qch}, for every $x\in X$, there is a neighborhood $K'$ (resp. $K''$) of $x$ which is a Noetherian Stein compactum, and an $O_{K'}(K')$-module $M'$ (resp.  $O_{K''}(K'')$-module $M''$) whose localization is $F'|_{K'}$ (resp. $F''|_{K''}$). By Fact \ref{ft:I10}, there is a neighborhood $K$ of $x\in K'^{\circ}\cap K''^{\circ}$  that is a Noetherian Stein compactum.  From \cite[\href{https://stacks.math.columbia.edu/tag/01BJ}{Tag 01BJ}]{stacks-project}, $F'|_K$ (resp. $F''|_K$) is associated with the $O_K(K)$-module $M'\otimes_{O_{K'}(K')}O_K(K)$ (resp. $M''\otimes_{O_{K''}(K'')}O_K(K)$).

Let $P=\Gamma(K,F)$.	By Lemma \ref{lm:locK} \ref{it:locsec} and \ref{it:locvanish}, the sequence (\ref{eq:extinQch}) induces a short exact sequence in $\Mod(O_K(K))$: \[0\to M'\otimes_{O_{K'}(K')}O_K(K)\to P\to M''\otimes_{O_{K''}(K'')}O_K(K)\to0.\] From Lemma \ref{lm:locK} \ref{it:fullfaithful}, by localization it induces a shot exact sequence in $\Mod(O_K)$: \[	0 \to \widetilde{M'\otimes_{O_{K'}(K')}O_K(K)} \to \tilde{P}\to \widetilde{M''\otimes_{O_{K''}(K'')}O_K(K)}\to0.\] By restricting to $K^{\circ}$  and \cite[\href{https://stacks.math.columbia.edu/tag/01BJ}{Tag 01BJ}]{stacks-project}, one has a commutative diagram
\begin{center}
	\begin{tikzcd}
		0 \arrow[r] & \widetilde{M'\otimes_{O_{K'}(K')}O_X(K^{\circ})} \arrow[r] \arrow[d] & \widetilde{P\otimes_{O_K(K)}O_X(K^{\circ})} \arrow[r] \arrow[d] & \widetilde{M''\otimes_{O_{K''}(K'')}O_X(K^{\circ})} \arrow[r] \arrow[d] & 0 \\
		0 \arrow[r] & F'|_{K^{\circ}} \arrow[r]                                       & F|_{K^{\circ}}  \arrow[r]        & F''|_{K^{\circ}}  \arrow[r]                                        & 0
	\end{tikzcd}
\end{center}  in $\Mod(O_{K^{\circ}})$.  The vertical morphisms are given by the canonical morphism $P\otimes_{O_K(K)}O_X(K^{\circ})\to \Gamma(K^{\circ},F)$ in $\Mod(O_X(K^{\circ}))$, and the adjunction of $\tilde{\cdot}:\Mod(O_X(K^{\circ}))\to \Mod(O_{K^{\circ}})$ and $\Gamma(K^{\circ},\cdot):\Mod(O_{K^{\circ}})\to \Mod(O_X(K^{\circ}))$. The rows are exact, and the two outside vertical arrows are isomorphisms. By the five lemma,  the middle vertical morphism is an isomorphism. By Example \ref{eg:inducedqch}, the $O_{K^{\circ}}$-module $ F|_{K^{\circ}}$ is quasi-coherent. Consequently, $F$ is quasi-coherent.

By \ref{it:kercokercls}, \ref{it:extensioncls} and \cite[\href{https://stacks.math.columbia.edu/tag/0754}{Tag 0754}]{stacks-project}, $\Qch(X)$ is a weak Serre subcategory of $\Mod(O_X)$.
\subsection*{Acknowledgments} I thank my supervisor, Anna Cadoret, for her constant support.
\printbibliography
\end{document}